\documentclass[reqno]{amsart}
\usepackage[american]{babel}
\usepackage{amssymb,euscript}
\numberwithin{equation}{section}
\newtheorem{theorem}{Theorem}[section]
\newtheorem{lemma}[theorem]{Lemma}

\theoremstyle{definition}
\newtheorem{definition}[theorem]{Definition}
\theoremstyle{remark}
\newtheorem{remark}[theorem]{Remark}

\DeclareMathOperator{\supp}{supp}

\begin{document}

\title[Zakharov--Kuznetsov Equation]
{Regular Solutions to Initial-Boundary Value Problems in a Half-Strip for Two-Dimensional Zakharov--Kuznetsov Equation}

\author[A.V.~Faminskii]{Andrei~V.~Faminskii}

\thanks{The publication was prepared with the support of the "RUDN University Program 5-100" and RFBR grants 17-01-00849, 17-51-52022, 18-01-00590.}

\address{Peoples' Friendship University of Russia (RUDN University), 6 Miklukho--Maklaya Street, Moscow, 117198, Russian Federation}

\email{afaminskii@sci.pfu.edu.ru}



\date{}

\begin{abstract}
Initial-boundary value problems in a half-strip with different types of boundary conditions for two-dimensional Zakharov--Kuznetsov equation are considered. Results on global well-posedness in classes of regular solutions in the cases of periodic and Neumann boundary conditions, as well as on internal regularity of solutions for all types of boundary conditions are established. Also in the case of Dirichlet boundary conditions one result on long-time decay of regular solutions is obtained.
\end{abstract}

\maketitle

\section{Introduction. Description of main results}\label{S1}
The paper is devoted to an initial-boundary value problem for two-dimensional Zakharov--Kuznetsov equation (ZK)
\begin{equation}\label{1.1}
u_t+bu_x+u_{xxx}+u_{xyy}+uu_x=0
\end{equation}
($b$ is a real constant), posed on half-strip $\Sigma_+=\mathbb R_+\times (0,L)=\{(x,y): x>0, 0<y<L\}$ with initial and boundary conditions
\begin{equation}\label{1.2}
u(0,x,y)=u_0(x,y),\qquad (x,y)\in\Sigma_+,
\end{equation}
\begin{equation}\label{1.3}
u(t,0,y)=\mu(t,y),\qquad (t,y)\in B_{T}=(0,T)\times (0,L),
\end{equation}
and boundary conditions for $(t,x)\in \Omega_{T,+}=(0,T)\times\mathbb R_+$ of one of the following four types: 
\begin{equation}\label{1.4}
\begin{split}
\mbox{whether}\qquad &a)\mbox{ } u(t,x,0)=u(t,x,L)=0,\\
\mbox{or}\qquad &b)\mbox{ } u_y(t,x,0)=u_y(t,x,L)=0,\\
\mbox{or}\qquad &c)\mbox{ } u(t,x,0)=u_y(t,x,L)=0,\\
\mbox{or}\qquad &d)\mbox{ } u \mbox{ is an  $L$-periodic function with respect to $y$.}
\end{split}
\end{equation}
The notation "problem \eqref{1.1}--\eqref{1.4}" is used for each of these four cases. We consider global solutions, so $T$ is an arbitrary positive number. The main results of the paper are related to well-posedness in classes of smooth solutions and to internal regularity of solutions. Weak and less regular solutions to the considered problem were previously studied in \cite{F18-1}.

ZK equation for the first time was derived in \cite{ZK} in the three-dimensional case for description of  ion-acoustic waves in magnetized plasma. In the considered two-dimensional case this equation is known now as a model of two-dimensional nonlinear waves in dispersive media, propagating in one preassigned ($x$) direction with deformations in the transverse ($y$) direction. A rigorous derivation of the ZK model can be found, for example, in \cite{H-K, LLS}. It is one of the variants of multi-dimensional generalizations of Korteweg--de~Vries equation (KdV) $u_t+bu_x+u_{xxx}+uu_x=0$.

For ZK equation there is a lot of literature, devoted to the initial value and initial-boundary value problems, where the variable $y$ is considered on the whole line (see, for example, bibliography in \cite{F18-1, F18-2} and recent papers \cite{K, Sh}). In particular, global well-posedness in Sobolev spaces $H^k$ of arbitrary large regularity to the initial value problem and the initial-boundary value problem, posed on $\mathbb R^2_+ = \{(x,y): x>0\}$ and $(0,1)\times \mathbb R$, was established (see, for example, \cite{F95, F08}). 

Initial-boundary value problems, where $y$ varies in a bounded interval, are less studied, however, from the physical point of view they seem at least the same important (\cite{LPS, STW, BF13, LT, L13, MP, DL, F15-2, F18-1, F18-2}). For example, there are no results on existence of  global solutions in Sobolev spaces with any prescribed regularity. In the present paper we obtain the corresponding results but only for the problems with boundary conditions of the cases b) and d). 

In comparison with such results, the gain of internal regularity of solutions does not depend on the type of boundary conditions. Starting with solutions, constructed in \cite{F18-1}, we establish results on any prescribed internal regularity depending on the properties of the initial function, in particular, on its decay rate as $x\to +\infty$. Note that for KdV equation first similar results were obtained in \cite{KF}. Gain of internal regularity of weak solutions to the initial value problem for ZK equation in the two-dimensional case was studied in \cite{FA, AF15}, for the initial-boundary value problem posed on $\mathbb R^2_+$ --- in \cite{AF18}. In the three-dimensional case for the initial value problem certain results on internal regularity of solutions were recently established in \cite{LP}.

Notation, used in the present paper, in many respects repeats the one from \cite{F18-1}. 
In what follows (unless stated otherwise) $i$, $j$, $k$, $l$, $m$, $n$ mean non-negative integers, $p\in [1,+\infty]$, $s\in\mathbb R$.  For any multi-index $\alpha=(\alpha_1,\alpha_2)$ let $\partial^\alpha =\partial^{\alpha_1}_{x}\partial^{\alpha_2}_{y}$, 
\begin{equation*}
|D^k\varphi|=\Bigl(\sum_{|\alpha|= k}(\partial^\alpha \varphi)^2\Bigr)^{1/2}, \qquad
|D\varphi|=|D^1\varphi|.
\end{equation*}
Let $L_{p,+}=L_p(\Sigma_+)$, $W_{p,+}^k=W_p^k(\Sigma_+)$, $H^s_+=H^s(\Sigma_+)$.

Introduce special function spaces taking into account boundary conditions \eqref{1.4}. Let 
$\Sigma= \mathbb R\times (0,L)$, 
$\widetilde{\EuScript S}(\overline{\Sigma})$ be a space of infinitely smooth on $\overline{\Sigma}$ functions $\varphi(x,y)$, such that $\displaystyle{(1+|x|)^n|\partial^\alpha\varphi(x,y)|\leq c(n,\alpha)}$ for any $n$, multi-index $\alpha$, $(x,y)\in \overline{\Sigma}$ and $\partial_y^{2m}\varphi\big|_{y=0} =\partial_y^{2m}\varphi\big|_{y=L}=0$ in the case a), $\partial_y^{2m+1}\varphi\big|_{y=0} =\partial_y^{2m+1}\varphi\big|_{y=L}=0$ in the case b), $\partial_y^{2m}\varphi\big|_{y=0} =\partial_y^{2m+1}\varphi\big|_{y=L}=0$ in the case c), $\partial_y^{m}\varphi\big|_{y=0} =\partial_y^{m}\varphi\big|_{y=L}$ in the case d) for any $m$.

Let $\widetilde H^s$ be the closure of $\widetilde{\EuScript S}(\overline{\Sigma})$ in the norm $H^s(\Sigma)$ and $\widetilde H_+^s$ be the restriction of $\widetilde H^s$ on $\Sigma_+$.

It is easy to see, that $\widetilde H^0_+=L_{2,+}$; for $j \geq 1$ in the case a) $\widetilde H^j_+=\{\varphi\in H^j_+: \partial_y^{2m}\varphi|_{y=0}=\partial_y^{2m}\varphi|_{y=L}=0, \ 2m<j\}$, in the case b) $\widetilde H^j_+=\{\varphi\in H^j_+: \partial_y^{2m+1}\varphi|_{y=0}=\partial_y^{2m+1}\varphi|_{y=L}=0,\ 2m+1<j\}$,  in the case d) $\widetilde H^j_+=\{\varphi\in H^j_+: \partial_y^{m}\varphi|_{y=0}=\partial_y^{m}\varphi|_{y=L}, \ m<j\}$.  

We also use an anisotropic Sobolev space $\widetilde H^{(0,k)}_+$ which is defined as the restriction on $\Sigma_+$ of a space $\widetilde H^{(0,k)}$, where the last space is the closure of $\widetilde{\EuScript S}(\overline{\Sigma})$ in the norm $\sum\limits_{m=0}^k \|\partial_y^m\varphi\|_{L_2(\Sigma)}$.

We say that $\rho(x)$ is an admissible weight function, if $\rho$ is an infinitely smooth positive function on $\overline{\mathbb R}_+$, such that $|\rho^{(j)}(x)|\leq c(j)\rho(x)$ for each natural $j$ and all $x\geq 0$. Note that such a function satisfies an inequality $\rho(x) \leq ce^{c_0 x}$ for certain positive constants $c_0$, $c$ and all $x\geq 0$. Any exponent $e^{2\alpha x}$ as well as $(1+x)^{2\alpha}$ are admissible weight functions. 

For an admissible weight function $\rho(x)$ let $\widetilde H^{k,\rho(x)}_+$ be a space of functions $\varphi(x,y)$, such that $\varphi\rho^{1/2}(x)\in \widetilde H^k_+$ (similar definitions for $\widetilde H^{(0,k),\rho(x)}_+$, $H^{k,\rho(x)}_+$). Let $L_{2,+}^{\rho(x)}= \widetilde H^{0,\rho(x)}_+= \{\varphi(x,y): \varphi\rho^{1/2}(x)\in L_{2,+}\}$. 

Let $\Pi_T^+ = (0,T)\times\Sigma_+$. Introduce the following spaces, in which we  consider solutions. 

\begin{definition}\label{D1.1} Let $X^{k,\rho(x)}(\Pi_T^+)$ for an admissible weight function $\rho(x)$, such that $\rho'(x)$ is also an admissible weight function, be a space of functions $u(t,x,y)$, such that  
\begin{equation}\label{1.5}
\partial_t^m u\in C([0,T]; \widetilde H_+^{k-3m,\rho(x)})\cap L_2(0,T;\widetilde H_+^{k-3m+1,\rho'(x)}).
\end{equation}
\end{definition}

Let $\psi_l(y)$, $l=1,2\dots$, be the orthonormal in $L_2(0,L)$ system of the eigenfunctions for the operator $(-\psi'')$ on the segment $[0,L]$ with corresponding boundary conditions   $\psi(0)=\psi(L)=0$ in the case a), $\psi'(0)=\psi'(L)=0$ in the case b), $\psi(0)=\psi'(L)=0$ in the case c), $\psi(0)=\psi(L),\psi'(0)=\psi'(L)$ in the case d), $\lambda_l$ be the corresponding eigenvalues. Such systems are well-known and can be written in trigonometric functions.

For description of properties of the boundary data $\mu$ introduce anisotropic functional spaces. Let $B=\mathbb R^t \times (0,L)$. Define the functional space $\widetilde{\EuScript S}(\overline{B})$ similarly to $\widetilde{\EuScript S}(\overline{\Sigma})$, where the variable $x$ is substituted by $t$. Let $\widetilde H^{s/3,s}(B)$ be the closure of $\widetilde{\EuScript S}(\overline{B})$ in the norm $H^{s/3,s}(B)$. 

More exactly, for any  $\mu\in \widetilde{\EuScript S}(\overline{B})$, $\theta\in\mathbb R$ and $l$ let
\begin{equation}\label{1.6}
\widehat\mu(\theta,l) \equiv \iint_B e^{-i\theta t}\psi_l(y)\mu(t,y)\,dtdy.
\end{equation}
Then the norm in $\widetilde H^{s/3,s}(B)$ is defined as $\Bigl(\sum\limits_{l=1}^{+\infty} 
\bigl\| (|\theta|^{2/3}+l^2)^{s/2}\widehat\mu(\theta,l)\bigr\|_{L_2(\mathbb R^\theta)}^2\Bigr)^{1/2}$ and the norm in $\widetilde H^{s/3,s}(I\times (0,L))$ for any interval $I\subset \mathbb R$ as the restriction norm. 

The use of these norm is justified by the following fact. Let $v(t,x,y)$ be the appropriate solution to the initial value problem
$$
v_t+v_{xxx}+v_{xyy}=0,\qquad v\big|_{t=0}=v_0.
$$
Then according to \cite{F08} uniformly with respect to $x\in \mathbb R$
\begin{equation}\label{1.7}
\bigl\|D_t^{1/3}v\bigr\|_{H_{t,y}^{s/3,s}(\mathbb R^2)}^2+
\bigl\|\partial_x v\bigr\|_{H_{t,y}^{s/3,s}(\mathbb R^2)}^2+
\bigl\|\partial_y v\bigr\|_{H_{t,y}^{s/3,s}(\mathbb R^2)}^2
\sim \|v_0\|_{H^s(\mathbb R^2)}^2
\end{equation}
(here $D^\alpha$ denotes the Riesz potential of the order $-\alpha$).

In \cite{F18-1} the following result on global well-posedness was established.

\begin{theorem}\label{T1.1}
Let $u_0\in \widetilde H^{3,\rho(x)}_+$ for an admissible weight function $\rho(x)$, such that $\rho'(x)$ is also an admissible weight function and $\rho'(x)\geq 1 \ \forall x\geq 0$. Let $\mu\in \widetilde H^{4/3,4}(B_T)$ for certain $T>0$, $\mu(0,y)\equiv u_0(0,y)$. Then problem \eqref{1.1}--\eqref{1.4} is well-posed in the space $X^{3,\rho(x)}(\Pi_T^+)$.
\end{theorem}

Now introduce the following auxiliary functions for compatibility conditions of the higher orders on the boundary data.

\begin{definition}\label{D1.2}
Let $\Phi_0(x,y)\equiv u_0(x,y)$ and for $m\geq 1$
\begin{multline}\label{1.8}
\Phi_m(x,y) \equiv -
(\partial_x^3+\partial_x\partial_y^2+b\partial_x)\Phi_{m-1}(x,y)\\-
\sum\limits_{l=0}^{m-1} {\binom {m-1}l} \Phi_l(x,y)\partial_x\Phi_{m-l-1}(x,y).
\end{multline}
\end{definition}

The first main theorem of the present paper is the following result on global well-posedness in the classes of regular solutions.

\begin{theorem}\label{T1.2}
Let the types b) or d) of boundary conditions \eqref{1.4} are considered. Let $u_0\in \widetilde H_+^{k,\rho(x)}$, $\mu \in \widetilde H^{(k+1)/3,k+1}(B_T)$ for certain $T>0$, natural $k\geq 4$, such that $k=3i$ or $k=3i+1$, $i\in\mathbb N$, and an admissible weight function $\rho(x)$, such that $\rho'(x)$ is also an admissible weight function and $\rho'(x)\geq 1 \ \forall x\geq 0$. Let $\partial^m_t \mu(0,y)\equiv \Phi_m(0,y)$ for $0\leq m<k/3$. Then problem \eqref{1.1}--\eqref{1.4} is well-posed in the space $X^{k,\rho(x)}(\Pi_T^+)$.
\end{theorem}

\begin{remark}\label{R1.1}
We mean that the problem is well-posed in $X^{k,\rho(x)}(\Pi_T^+)$, if there exists a unique solution $u(t,x,y)$ in this space and the map $(u_0,\mu) \mapsto u$ is Lipschitz continuous on any ball in the norm of the map $H_+^{k,\rho(x)} \times H^{(k+1)/3,k+1}(B_T)$ into $X^{k,\rho(x)}(\Pi_T^+)$.
\end{remark}

\begin{remark}\label{R1.2}
According to \eqref{1.7} the assumptions on the boundary data $\mu$ are natural. Both the exponential weight $\rho(x)\equiv \frac1{2\alpha}e^{2\alpha x}$, $\alpha>0$, and the power weight $\rho(x)\equiv \frac1{2\alpha}(1+x)^{2\alpha}$, $\alpha\geq 1/2$,  satisfy the hypothesis of the theorem.  
\end{remark}

Introduce certain additional notation to formulate results on internal regularity. For any $x_0\geq 0$ let $\Sigma_{x_0} = (x_0,+\infty)\times (0,L)$, $\Pi_T^{x_0} = (0,T)\times \Sigma_{x_0}$ (then $\Sigma_+=\Sigma_0$, $\Pi_T^+ =\Pi_T^0$). For any $y_0\in [0,L/2)$ let $\Sigma_{x_0,y_0} = (x_0,+\infty)\times (y_0, L-y_0)$, $\Pi_T^{x_0,y_0} = (0,T)\times \Sigma_{x_0,y_0}$ (then $\Sigma_{x_0}=\Sigma_{x_0,0}$, $\Pi_T^{x_0} =\Pi_T^{x_0,0}$).

Let $L_{2,x_0}^{\rho(x)}= \{\varphi(x,y): \varphi\rho^{1/2}(x)\in L_{2,x_0}\}$, $L_{2,x_0,y_0}^{\rho(x)}= \{\varphi(x,y): \varphi\rho^{1/2}(x)\in L_{2,x_0,y_0}\}$ (then $L_{2,+}^{\rho(x)}= L_{2,0}^{\rho(x)}$, $L_{2,x_0}^{\rho(x)}= L_{2,x_0,0}^{\rho(x)}$).

Let $\widetilde H_{x_0}^{k,\rho(x)} = \{\varphi(x,y): \varphi\rho^{1/2}(x)\in \widetilde H_{x_0}^{k}\}$, where $\widetilde H_{x_0}^{k}$ is the restriction of $\widetilde H^{k}$ on $\Sigma_{x_0}$.

Finally, let $X^{k,\rho(x)}(\Pi_T^{x_0})$ for an admissible weight function $\rho(x)$, such that $\rho'(x)$ is also an admissible weight function, be a space consisting of functions $u(t,x,y)$, such that  
$$
\partial_t^m u\in C([0,T]; \widetilde H_{x_0}^{k-3m,\rho(x)})\cap L_2(0,T;\widetilde H_{x_0}^{k-3m+1,\rho'(x)}).
$$

\begin{theorem}\label{T1.3}
Let the hypothesis of Theorem~\ref{T1.1} be satisfied. Let, in addition, $\partial^n_x u_0 \in \widetilde H_{x_0}^{3,\rho(x)}$ for all $x_0>0$ and certain natural $n$, $\rho^{1/2}(x)\leq c\rho'(x)\ \forall x\geq 0$ and certain constant $c$. Consider the unique solution $u(t,x,y)$ to problem \eqref{1.1}--\eqref{1.4} from the space $X^{3,\rho(x)}(\Pi_T^+)$, then $\partial_x^n u\in X^{3,\rho(x)}(\Pi_T^{x_0})$ for all $x_0>0$.
\end{theorem}

\begin{remark}\label{R1.3}
Both the exponential weight $\rho(x)\equiv \frac1{2\alpha}e^{2\alpha x}$, $\alpha>0$, and the power weight $\rho(x)\equiv \frac1{2\alpha}(1+x)^{2\alpha}$, $\alpha\geq 1$,  satisfy the hypothesis of the theorem.  
\end{remark}

\begin{theorem}\label{T1.4}
Let the hypothesis of Theorem~\ref{T1.3} be satisfied. Let $\rho_j(x)$, $0\leq j\leq n$, be a set of admissible weight functions, such that all $\rho'_j$ are also admissible weight functions, $\rho_0(x)\equiv \rho(x)$, $\rho_j(x) \leq c\bigl(\rho'_j(x) \rho'_{j-1}(x)\bigr)^{1/2}$ for $j\geq 1$, all $x\geq 0$ and a certain positive constant $c$. Let $\partial_x^{m-j} \partial_y^{j+3} u_0 \in L_{2,x_0,y_0}^{\rho_j(x)}$ for all $x_0>0$, $y_0\in (0,L/2)$, $1\leq m\leq n$, $1\leq j\leq m$. Then for all $x_0>0$, $y_0\in (0,L/2)$ and $1\leq j \leq n$
\begin{gather*}
\partial_x^{n-j}\partial_y^{j+3} u \in C\bigl([0,T];L_{2,x_0,y_0}^{\rho_j(x)}\bigr),\\
\partial_x^{n-j}\partial_y^{j+4} u \in L_2\bigl(0,T;L_{2,x_0,y_0}^{\rho'_j(x)}\bigr).
\end{gather*}
\end{theorem}

\begin{remark}\label{R1.4}
Any exponential weight $\rho(x)\equiv \frac1{2\alpha}e^{2\alpha x}$, $\alpha>0$, verifies the hypothesis of the theorem ($\rho_j(x)=\rho(x)\ \forall j$). The power weight $\rho(x) \equiv \frac1{2\alpha}(1+x)^{2\alpha}$ for $\alpha > n$ verifies the hypothesis of the theorem ($\rho_j(x)= \frac1{2\alpha}(1+x)^{2(\alpha-j)}$).
\end{remark}

\begin{remark}\label{R1.5}
Note that Theorems~\ref{T1.3} and~\ref{T1.4} are valid for all types of boundary conditions \eqref{1.4}.
\end{remark}

Now we present one result of large-time decay of small solutions in the weighed $H^3$-norm for the cases a) and c). Here we use only exponential weights. It is based on some ideas from \cite{LT, L13}, where similar results were obtained in the exponentially weighted $L_2$ and $H^1$-norms (see also \cite{F18-1}). For the problem on a rectangle large-time decay of small solutions in $H^2$-norm was established in \cite{L16}. 

\begin{theorem}\label{T1.5}
Let the types a) or c) of boundary conditions \eqref{1.4} are considered. Let $L_0=+\infty$ if $b\leq 0$, and if $b>0$ there exists $L_0>0$, such that in both cases for any $L\in (0,L_0)$ there exist $\alpha_0>0$, $\epsilon_0>0$ and $\beta>0$, such that if $u_0\in \widetilde H_+^{3,\exp(2\alpha x)}$ for $\alpha\in (0,\alpha_0]$, $u_0(0,y)\equiv 0$, $\|u_0\|_{L_{2,+}}\leq\epsilon_0$, $\mu\equiv 0$,  the corresponding unique solution $u(t,x,y$) to problem \eqref{1.1}--\eqref{1.4} from the space $X^{3,exp(2\alpha x)}(\Pi_T^+)$ $\forall T>0$ satisfies an inequality
\begin{equation}\label{1.9}
\|e^{\alpha x}u(t,\cdot,\cdot)\|^2_{H_+^3} + \|e^{\alpha x}u_t(t,\cdot,\cdot)\|^2_{L_{2,+}}\leq ce^{-\alpha\beta t}\qquad \forall t\geq 0,
\end{equation}
where the constant $c$ depends on $b$, $\alpha$, $\beta$, $\|u_0\|_{ H_+^{3,exp(2\alpha x)}}$.
\end{theorem}

Further, let $\eta(x)$ denotes a cut-off function, namely, $\eta$ is an infinitely smooth non-decreasing function on $\mathbb R$  such that $\eta(x)=0$ when $x\leq 0$, $\eta(x)=1$ when $x\geq 1$, $\eta(x)+\eta(1-x)\equiv 1$.

We drop limits of integration in integrals over the whole half-strip $\Sigma_+$.

\medskip

In our study we use the following interpolating inequality from \cite{F18-1}.
If $\rho_1(x)$, $\rho_2(x)$ are two admissible weight functions, such that $\rho_1(x)\leq c_0\rho_2(x)$ $\forall x\geq 0$ for some constant $c_0>0$, then there exists a constant $c>0$, such that for every function $\varphi(x,y)$, satisfying $|D\varphi|\rho_1^{1/2}(x)\in L_{2,+}$, $\varphi\rho_2^{1/2}(x)\in L_{2,+}$,
\begin{equation}\label{1.10}
\bigl\| \varphi\rho_1^{1/4}(x)\rho_2^{1/4}(x)\bigr\|_{L_{4,+}} \leq c
\bigl\| |D\varphi|\rho_1^{1/2}(x)\bigr\|^{1/2}_{L_{2,+}}
\bigl\| \varphi\rho_2^{1/2}(x)\bigr\|^{1/2}_{L_{2,+}}
+c\bigl\|\varphi\rho_2^{1/2}(x)\bigr\|_{L_{2,+}}.
\end{equation}
If $\varphi\big|_{y=0}=0$ or $\varphi\big|_{y=L}=0$, then the constant $c$ in \eqref{1.10} is uniform with respect to $L$.

We also use the following obvious  interpolating inequalities: 
\begin{equation}\label{1.11}
\int_0^L \varphi^2\big|_{x=0}\,dy \leq c \Bigl(\iint \varphi_x^2\rho'\,dxdy\Bigr)^{1/2}
\Bigl(\iint \varphi^2\rho\,dxdy\Bigr)^{1/2}+
c\iint \varphi^2\rho\,dxdy
\end{equation}
and
\begin{equation}\label{1.12}
\|\varphi\rho^{1/2} \|_{L_{\infty,+}} \leq c \|\varphi\|_{H^{2,\rho(x)}_+}
\end{equation}
(the constants $c$ depend on the properties of an admissible weight function $\rho$).

For the decay results we need Steklov inequalities in the following form: 
\begin{equation}\label{1.13}
\int_0^L \psi^2(y)\,dy \leq \frac{\sigma L^2}{\pi^2} \int_0^L \bigl(\psi'(y)\bigr)^2\,dy,
\end{equation}
where $\sigma=1$ if $\psi\in H_0^1(0,L)$, $\sigma=4$ if $\psi\in H^1(0,L)$, $\psi\big|_{y=0}=0$.

We also use the following interpolating inequality (see, for example, \cite{F18-1}): for any admissible weight function $\rho(x)$ if $\varphi\in H_+^{1,\rho(x)}$, $\varphi_{xxx} =\varphi_0+ \varphi_{1x}$, where $\varphi_0, \varphi_1 \in L_{2,+}^{\rho(x)}$, then $\varphi\in H_+^{2,\rho(x)}$ and
\begin{equation}\label{1.14}
\|\varphi_{xx}\|_{L_{2,+}^{\rho(x)}} \leq c(\rho)\bigl(\|\varphi_{0}\|_{L_{2,+}^{\rho(x)}} + \|\varphi_{1}\|_{L_{2,+}^{\rho(x)}} +\|\varphi\|_{H_+^{1,\rho(x)}}\bigr).
\end{equation}

It follows from \cite{BIN} and properties of the admissible weight function $\rho(x)$, that
\begin{equation}\label{1.15}
\|\varphi_{xy}\|_{L_{2,+}^{\rho(x)}} \leq c(\rho)\Bigl(\|\varphi_{xx}\|_{L_{2,+}^{\rho(x)}} +\|\varphi_{yy}\|_{L_{2,+}^{\rho(x)}} +\|\varphi\|_{L_{2,+}^{\rho(x)}}\Bigr).
\end{equation}

The paper is organized as follows. Auxiliary linear problems are considered in Section~\ref{S2}.  Section~\ref{S3} is devoted to the well-posedness results for the original problems in regular classes. Results on internal regularity are proved in Section~\ref{S4}. Decay of solutions is studied in Section~\ref{S5}.

\section{Auxiliary linear problems}\label{S2}

Consider an initial-boundary value in $\Pi_T^+$ for a linear equation
\begin{equation}\label{2.1}
u_t+bu_x+u_{xxx}+u_{xyy}=f(t,x,y)
\end{equation}
with initial and boundary conditions \eqref{1.2}--\eqref{1.4}. It was shown in \cite{F18-1}, that weak solutions to this problem are unique in the space $L_2(\Pi_T^+)$.

We introduce certain additional function space. Let $\EuScript S_{exp}(\overline{\Sigma}_+)$ denotes a space of infinitely smooth functions $\varphi(x,y)$ on $\overline{\Sigma}_+$,  such that $e^{n x}|\partial^\alpha\varphi(x,y)|\leq c(n,\alpha)$ for any $n$, multi-index $\alpha$, $(x,y)\in \overline{\Sigma}_+$. 

We start with two simple technical assertions.

\begin{lemma}\label{L2.1}
Let $v\in C^\infty([0,T]; \EuScript S_{exp}(\overline{\Sigma}_+))$ and satisfy the following boundary conditions: $v\big|_{x=0}=0$ and 
\begin{equation}\label{2.2}
\begin{split}
\mbox{either}\quad &1)\mbox{ } v\big|_{y=0}=0\ \mbox{or } v_y\big|_{y=0}=0,\quad
v\big|_{y=L}=0\ \mbox{or } v_y\big|_{y=L}=0,\\
\mbox{or}\qquad &2)\mbox{ } v\big|_{y=0}= v\big|_{y=L},\quad
v_y\big|_{y=0}= v_y\big|_{y=L}.
\end{split}
\end{equation}
Let 
\begin{equation}\label{2.3}
F\equiv v_t +bv_x +v_{xxx} +v_{xyy}.
\end{equation}
Then for any admissible weight function $\rho(x)$
\begin{multline}\label{2.4}
\frac{d}{dt} \iint v^2\rho\,dxdy + \rho(0)\int_0^L v_x^2\big|_{x=0}\,dy +
\iint (3v_x^2+v_y^2-bv^2)\rho'\,dxdy \\ -
\iint v^2\rho'''\,dxdy = 2\iint Fv\rho\,dxdy,
\end{multline}
\begin{multline}\label{2.5}
\frac{d}{dt} \iint (v_x^2+v_y^2)\rho\,dxdy +
\int_0^L (v_{xx}^2\rho +2v_{xx}v_x\rho' -v_x^2\rho'' +bv_x^2\rho)\big|_{x=0}\,dy \\ + \iint (3v_{xx}^2 +4v_{xy}^2 +v_{yy}^2- bv_x^2 -bv_y^2)\rho'\,dxdy -
\iint (v_x^2 +v_y^2)\rho'''\,dxdy \\ =
-2\iint F (v_{xx}\rho +v_x \rho' +v_{yy}\rho)\,dxdy.
\end{multline}
\end{lemma}

\begin{proof}
The proof is performed via multiplication of equality \eqref{2.3} correspondingly by $2v\rho$ and $-2\bigl((v_x\rho)_x +v_{yy}\rho\bigr)$ and consequent integration.
\end{proof}

\begin{lemma}\label{L2.2}             
Let $v\in C^\infty([0,T]; \EuScript S_{exp}(\overline{\Sigma}_+))$ and satisfy boundary conditions \eqref{2.2}. Then for any admissible weight function $\rho(x)$, any $x_0>0$ and $\eta_{x_0}(x) \equiv \eta\bigl((2x-x_0)/x_0\bigr)$
\begin{multline}\label{2.6}
\frac{d}{dt} \iint v^2\rho\eta_{x_0}\,dxdy + 
\iint (3v_x^2+v_y^2-bv^2)(\rho\eta_{x_0})'\,dxdy \\ -
\iint v^2(\rho\eta_{x_0})'''\,dxdy = 2\iint Fv\rho\eta_{x_0}\,dxdy,
\end{multline}
\begin{multline}\label{2.7}
\frac{d}{dt} \iint (v_x^2+v_y^2)\rho\eta_{x_0}\,dxdy +
\iint (3v_{xx}^2 +4v_{xy}^2 +v_{yy}^2- bv_x^2 -bv_y^2)(\rho\eta_{x_0})'\,dxdy  \\ -
\iint (v_x^2 +v_y^2)(\rho\eta_{x_0})'''\,dxdy =
-2\iint F (v_{xx}\rho\eta_{x_0} +v_x (\rho\eta_{x_0})' +v_{yy}\rho\eta_{x_0})\,dxdy,
\end{multline}
where the function $F$ is given by formula \eqref{2.3}.
\end{lemma}

\begin{proof}
The proof is performed via multiplication of equality \eqref{2.3} correspondingly by $2v\rho\eta_{x_0}$ and $-2\bigl((v_x\rho\eta_{x_0})_x +v_{yy}\rho\eta_{x_0}\bigr)$ and consequent integration.
\end{proof}

Now we introduce the function spaces to describe properties of the right side of equation \eqref{2.1}.

\begin{definition}\label{D2.1}
Let $\rho(x)$ be an admissible weight function, such that $\rho'(x)$ is also admissible. 
For $k=3i$, $i\geq 0$, define a space $Y^{k,\rho(x)}(\Pi_T^+)$, consisting of functions $f(t,x,y)$, such that $f\equiv f_0 +f_{1x}$, where
\begin{gather*}
\partial_t^i f_0 \in L_1(0,T;L_{2,+}^{\rho(x)}), \quad
\partial_t^i f_1 \in L_2(0,T;L_{2,+}^{\rho^2(x)/\rho'(x)}), \\
\partial_t^m f \in L_2(0,T; \widetilde H_+^{(0,k-3m-1), \rho^2(x)/\rho'(x)})\quad \mbox{for}\ m\leq i-1\quad \mbox{if}\ i\geq 1,\\
f\in X^{k-3,\rho(x)}(\Pi_T^+)\quad \mbox{if}\ i\geq 1,
\end{gather*}
endowed with the natural norm. For $k=3i+1$, $i\geq 0$, define a space $Y^{k,\rho(x)}(\Pi_T^+)$, consisting of functions $f(t,x,y)$, such that
\begin{gather*}
\partial_t^m f \in L_2(0,T; \widetilde H_+^{(0,k-3m-1), \rho^2(x)/\rho'(x)})\quad \mbox{for}\ m\leq i,\\
f\in X^{k-3,\rho(x)}(\Pi_T^+)\quad \mbox{if}\ i\geq 1,
\end{gather*}
endowed with the natural norm.
\end{definition}

\begin{definition}\label{D2.2}
Let $\widetilde\Phi_0(x,y)\equiv u_0(x,y)$ and for $m\geq 1$
\begin{equation}\label{2.8}
\widetilde\Phi_m(x,y) \equiv \partial_t^{m-1} f(0,x,y) -
(\partial_x^3+\partial_x\partial_y^2+b\partial_x)\widetilde\Phi_{m-1}(x,y).
\end{equation}
\end{definition}

\begin{lemma}\label{L2.3}
Let $\rho(x)$ be an admissible weight function such that $\rho'(x)$ is also admissible. Let either $k=3i$ or $k=3i+1$, $i\geq 0$, $u_0\in \widetilde H_+^{k,\rho(x)}$, $\mu\equiv 0$, $f\in Y^{k,\rho(x)}(\Pi_T^+)$, $\widetilde\Phi_m(0,y)\equiv 0$ for $m< k/3$. Then there exists a unique solution to problem \eqref{2.1}, \eqref{1.2}--\eqref{1.4} $u\in X^{k,\rho(x)}(\Pi_T^+)$ and for any $t_0\in (0,T]$
\begin{equation}\label{2.9}
\|u\|_{X^{k,\rho(x)}(\Pi_{t_0}^+)} \leq c(T)\left( \|u_0\|_{H_+^{k,\rho(x)}}+ 
\|f\|_{Y^{k,\rho(x)}(\Pi_{t_0}^+)}\right).
\end{equation}
\end{lemma}

\begin{proof}
Without loss of generality assume that $u_0\in \widetilde{\EuScript S}(\overline{\Sigma})\cap \EuScript S_{exp}(\overline{\Sigma}_+)$, $f\in C^\infty([0,2T]; \widetilde{\EuScript S}(\overline{\Sigma})\cap\EuScript S_{exp}(\overline{\Sigma}_+))$, $\widetilde\Phi_m(0,y)\equiv 0$ for all $m$ and consider solutions $u \in C^\infty([0,T]; .
\widetilde{\EuScript S}(\overline{\Sigma})\cap\EuScript S_{exp}(\overline{\Sigma}_+))$, constructed in \cite{F18-1}.

Note that
\begin{equation}\label{2.10}
\widetilde\Phi_m = (-1)^m(\partial_x^3+\partial_x\partial_y^2+b\partial_x)^m u_0 +
\sum\limits_{l=0}^{m-1} (-1)^{m-l-1}(\partial_x^3+\partial_x\partial_y^2+b\partial_x)^{m-l-1} \partial_t^l f\big|_{t=0}
\end{equation}
and, thus, for $m\leq k/3$
\begin{equation}\label{2.11}
\|\widetilde\Phi_m\|_{H^{k-3m,\rho(x)}} \leq c \Bigl(\|u_0\|_{H_+^{k,\rho(x)}} +
\sum\limits_{l=0}^{m-1} \|\partial_t^l f\|_{C([0,t_0];H_+^{k-3(l+1),\rho(x)})}\Bigr).
\end{equation}
Moreover, 
\begin{equation}\label{2.12}
\partial_t^m u\big|_{t=0} = \widetilde\Phi_m.
\end{equation}

In the case $k=3i$ apply for $v\equiv \partial_t^i u$ equality \eqref{2.4}, then since $F= \partial_t^i f_0 + \partial_t^i f_{1x}$ and for an arbitrary $\varepsilon>0$
\begin{multline*}
2\iint \partial_t^i f_{1x} v\rho\,dxdy =- 2\iint \partial_t^i f_1 (v\rho)_x\,dxdy \\ \leq
\varepsilon \iint (v_x^2 +v^2)\rho'\,dxdy +c(\varepsilon) \iint (\partial_t^i f_1)^2 \frac{\rho^2}{\rho'}\,dxdy
\end{multline*}
with the use of \eqref{2.11} and \eqref{2.12} one can derive that
\begin{equation}\label{2.13}
\|\partial_t^i u\|_{C([0,t_0];L_{2,+}^{\rho(x)})} + \|\partial_t^i u\|_{L_2(0,t_0;H_+^{1,\rho'(x)})} \leq c(T)\left( \|u_0\|_{H_+^{k,\rho(x)}}+ \|f\|_{Y^{k,\rho(x)}(\Pi_{t_0}^+)}\right).
\end{equation}

Next, for $v\equiv \partial_t^m\partial_y^{k-3m-1} u$, where $m\leq i-1$ in the case $k=3i$ and $m\leq i$ in the case $k=3i+1$, apply equality \eqref{2.5}, then since
\begin{multline*}
-2\iint \partial_t^m\partial_y^{k-3m-1} f (v_{xx}\rho +v_x \rho' +v_{yy}\rho)\,dxdy  \\ \leq \varepsilon \iint (v_{xx}^2 +v_{yy}^2 +v_x^2)\rho'\,dxdy +
c(\varepsilon) \iint (\partial_t^m\partial_y^{k-3m-1} f)^2\frac{\rho^2}{\rho'}\,dxdy
\end{multline*} 
and by virtue of \eqref{1.11}
$$
\int_0^L v_x^2\big|_{x=0}\,dy \leq \varepsilon \iint v_{xx}^2\rho'\,dxdy +
c(\varepsilon)\iint v_x^2\rho\,dxdy
$$
similarly to \eqref{2.13} the following estimate holds:
\begin{multline}\label{2.14}
\|\partial_t^m \partial_y^{k-3m-1} u\|_{C([0,t_0];H_+^{1,\rho(x)})} + 
\|\partial_t^m \partial_y^{k-3m-1} u\|_{L_2(0,t_0;H_+^{2,\rho(x)})} \\ \leq c(T)\left( \|u_0\|_{H_+^{k,\rho(x)}}+ \|f\|_{Y^{k,\rho(x)}(\Pi_{t_0}^+)}\right).
\end{multline}

In particular, the desired result is already proved for $k=0$ and $k=1$ (in fact, it was obtained in \cite{F18-1}). In the next step we assume that for the smaller values of $k$ the result is established and derive with use of induction with respect to $l=i-m$ the following inequality: for $j\leq k-3m$
\begin{multline}\label{2.15}
\|\partial_t^m\partial_x^j\partial_y^{k-3m-j} u\|_{C([0,t_0];L_{2,+}^{\rho(x)})} +
\|\partial_t^m\partial_x^j\partial_y^{k-3m-j} u\|_{L_2(0,t_0;H_+^{1,\rho(x)})} \\
\leq c(T)\left( \|u_0\|_{H_+^{k,\rho(x)}}+ \|f\|_{Y^{k,\rho(x)}(\Pi_{t_0}^+)}\right).
\end{multline}
Note that for $m=i$ and $m<i$, $j\leq 1$ it succeeds from \eqref{2.13} and \eqref{2.14}. If $m<i$ and $j\geq 3$ then equality \eqref{2.1} yields that
$$
\partial_t^m\partial_x^j\partial_y^{k-3m-j} u = \partial_t^m\partial_x^{j-3}\partial_y^{k-3m-j} (f-u_t-bu_x-u_{xyy}).
$$
In particular, we obtain \eqref{2.15} for $j=3$. Application of \eqref{1.15} yields that \eqref{2.15} succeeds for $j=2$. Finally, use induction with respect to $j$.
\end{proof}

The next lemma will be used in the last section.
\begin{lemma}\label{L2.4}
Let the hypothesis of Lemma~\ref{L2.3} be satisfied for $k=3$. Then there exist functions $\nu_1, \nu_2\in L_2(0,T)$, such that for the corresponding solution to problem \eqref{2.1}, \eqref{1.2}--\eqref{1.4} and a.e. $t\in (0,T)$
\begin{multline}\label{2.16}
\frac{d}{dt} \iint u_t^2\rho\,dxdy + \rho(0)\int_0^L \nu_1^2\,dy +
\iint (3u_{tx}^2+u_{ty}^2-bu_t^2)\rho'\,dxdy \\ -
\iint u_t^2\rho'''\,dxdy = 2\iint f_t u_t\rho\,dxdy,
\end{multline}
\begin{multline}\label{2.17}
\frac{d}{dt} \iint (u_{xy}^2+u_{yy}^2)\rho\,dxdy +
\int_0^L (u_{xxy}^2\rho +2u_{xxy}u_{xy}\rho' -u_{xy}^2\rho'' +bu_{xy}^2\rho)\big|_{x=0}\,dy \\ + \iint (3u_{xxy}^2 +4u_{xyy}^2 +u_{yyy}^2- bu_{xy}^2 -bu_{yy}^2)\rho'\,dxdy -
\iint (u_{xy}^2 +u_{yy}^2)\rho'''\,dxdy \\ =
-2\iint f_y (u_{xxy}\rho +u_{xy} \rho' +u_{yyy}\rho)\,dxdy.
\end{multline}
\begin{multline}\label{2.18}
\frac{d}{dt} \iint (u_{xyy}^2+u_{yyy}^2)\rho\,dxdy +
\int_0^L (\nu_2^2\rho +2\nu_2u_{xyy}\rho' -u_{xyy}^2\rho'' +bu_{xyy}^2\rho)\big|_{x=0}\,dy \\ + \iint (3u_{xxyy}^2 +4u_{xyyy}^2 +u_{yyyy}^2- bu_{xyy}^2 -bu_{yyy}^2)\rho'\,dxdy -
\iint (u_{xyy}^2 +u_{yyy}^2)\rho'''\,dxdy \\ =
-2\iint f_{yy} (u_{xxyy}\rho +u_{xyy} \rho' +u_{yyyy}\rho)\,dxdy.
\end{multline}
\end{lemma}

\begin{proof}
In the smooth case equality \eqref{2.16} coincide with \eqref{2.4} for $v\equiv u_t$, $\nu_1\equiv u_{tx}\big|_{x=0}$, equality \eqref{2.17} --- with \eqref{2.5} for $v\equiv u_{xy}$, equality \eqref{2.18} --- with \eqref{2.5} for $v\equiv u_{xyy}$, $\nu_2\equiv 
u_{xxyy}\big|_{x=0}$ and in the general case is obtained via closure.
\end{proof}

\begin{lemma}\label{L2.5}
Let the hypothesis of Lemma~\ref{L2.3} be satisfied for $k=3$. Assume, in addition, that for all $x_0>0$ and certain natural $n$
$$
\partial_x^n u_0 \in \widetilde H_{x_0}^{3,\rho(x)},\qquad
\partial_x^n f \in L_2(0,T; \widetilde H_{x_0}^{2,\rho^2(x)/\rho'(x)}).
$$
Then $\partial_x^n u \in X^{3,\rho(x)}(\Pi_T^{x_0})$ for all $x_0>0$.
\end{lemma}

\begin{proof}
Consider first smooth solutions as in the proof of Lemma~\ref{L2.3}. Apply equality \eqref{2.7} for $v\equiv \partial_x^{n+2-j}\partial_y^j u$, $j\leq 2$, then since $\supp \eta'_{x_0}\subset [x_0/2,x_0]$ similarly to \eqref{2.14} we derive the following inequality:
\begin{multline}\label{2.19}
\sum\limits_{j=0}^3 \|\partial_x^{n+3-j}\partial_y^j 
u\|_{C([0,T];L_{2,x_0}^{\rho(x)})} +
\sum\limits_{j=0}^4 \|\partial_x^{n+4-j}\partial_y^j 
u\|_{L_2(0,T;L_{2,x_0}^{\rho'(x)})} \\ \leq 
c(x_0)\Bigl(\sum\limits_{j=0}^3 \|\partial_x^{n+3-j}\partial_y^j u_0\|_{L_{2,x_0/2}^{\rho(x)}} +
\sum\limits_{j=0}^2 \|\partial_x^{n+2-j}\partial_y^j 
f\|_{L_2(0,T;L_{2,x_0/2}^{\rho^2(x)/\rho'(x)})} \\ +
\sum\limits_{j=0}^3  \|\partial_x^{n+3-j}\partial_y^j 
u\|_{L_2\bigl((0,T)\times (x_0/2, x_0) \times (0,L)\bigr)}\Bigr).
\end{multline}
Note that for $n=1$ the last term in the right side of \eqref{2.19} is already estimated since $u\in L_2(0,T; H_+^{4,\rho'(x)})$. Then induction with respect to $n$ and closure provide the desired result.
\end{proof}

\begin{lemma}\label{L2.6}
Let the hypothesis of Lemma~\ref{L2.5} be satisfied. Let $\rho_j(x)$, $0\leq j\leq n$, be the same set of functions as in the hypothesis of Theorem~\ref{T1.4}. Let for all $x_0>0$, $y_0\in (0,L/2)$, $1\leq m\leq n$, $1\leq j\leq m$
$$
\partial_x^{m-j} \partial_y^{j+3} u_0 \in L_{2,x_0,y_0}^{\rho_j(x)},\quad
\partial_x^{m-j} \partial_y^{j+2} f \in L_2(0,T;L_{2,x_0,y_0}^{\rho'_{j-1}(x)}).
$$
Then for all $x_0>0$, $y_0\in (0,L/2)$ and $1\leq j \leq n$
\begin{gather*}
\partial_x^{n-j}\partial_y^{j+3} u \in C\bigl([0,T];L_{2,x_0,y_0}^{\rho_j(x)}\bigr),\\
\partial_x^{n-j}\partial_y^{j+4} u \in L_2\bigl(0,T;L_{2,x_0,y_0}^{\rho'_j(x)}\bigr).
\end{gather*}
\end{lemma}

\begin{proof}
Again first consider smooth solutions. Let $\psi_{y_0}(y) \equiv \eta\bigl((2y-y_0)/y_0\bigr)\eta\bigl((2L-2y-y_0)/y_0\bigr)$. Apply equality \eqref{2.6} for $v\equiv \partial_x^{n-j}\partial_y^{j+3}u\psi_{y_0}$, $\rho(x)\equiv \rho_j(x)$, $1\leq j\leq n$, then 
$$
F\equiv \partial_x^{n-j}\partial_y^{j+3}f\psi_{y_0} + 
2\partial_x^{n-j+1}\partial_y^{j+4}u\psi'_{y_0} +
\partial_x^{n-j+1}\partial_y^{j+3}u\psi''_{y_0}.
$$
Note that
$$
v_y^2 \geq \frac12 (\partial_x^{n-j}\partial_y^{j+4}u\psi_{y_0})^2 -
(\partial_x^{n-j}\partial_y^{j+3}u\psi'_{y_0})^2,
$$
\begin{multline*}
2\int_0^L Fv\,dy = -2\int_0^L \partial_x^{n-j}\partial_y^{j+2} f (\partial_x^{n-j}\partial_y^{j+3}u\psi^2_{y_0})_y\,dy \\
-4\int_0^L \partial_x^{n-j+1}\partial_y^{j+3}u \partial_x^{n-j}\partial_y^{j+4}u\psi_{y_0} \psi'_{y_0} \\-
\int_0^L \partial_x^{n-j+1}\partial_y^{j+3}u \partial_x^{n-j}\partial_y^{j+3}u
\bigl(4(\psi'_{y_0})^2 +2\psi_{y_0}\psi''_{y_0}\bigr)\,dy,
\end{multline*}
where
\begin{multline*}
-4\iint \partial_x^{n-j+1}\partial_y^{j+3}u \partial_x^{n-j}\partial_y^{j+4}u\psi_{y_0} \psi'_{y_0}\rho_j\eta_{x_0}\,dxdy  \\ \leq \varepsilon
\iint (\partial_x^{n-j}\partial_y^{j+4}u\psi_{y_0})^2 \rho'_j\eta_{x_0}\,dxdy  \\ +
c(\varepsilon) \iint (\partial_x^{n-j+1}\partial_y^{j+3}u\psi'_{y_0})^2 \frac{\rho_j^2}{\rho'_j}\eta_{x_0}\,dxdy.
\end{multline*}
As a result, since $\rho_j^2/\rho'_j \leq c \rho'_{j-1}$
\begin{multline}\label{2.20}
\|\partial_x^{n-j}\partial_y^{j+3} u\|_{C([0,T];L_{2,x_0,y_0}^{\rho_j(x)})} +
\|\partial_x^{n-j}\partial_y^{j+4} u\|_{L_2(0,T;L_{2,x_0,y_0}^{\rho'_j(x)})} \\ \leq
\|\partial_x^{n-j}\partial_y^{j+3} u_0\|_{L_{2,x_0/2,y_0/2}^{\rho_j(x)}} + c(x_0,y_0) \Bigl(\|\partial_x^{n-j}\partial_y^{j+2} f\|_{L_2(0,T;L_{2,x_0/2,y_0/2}^{\rho'_j(x)})} \\+
\|\partial_x^{n-j+1}\partial_y^{j+3} u\|_{L_2(0,T;L_{2,x_0/2,y_0/2}^{\rho'_{j-1}(x)})} +
\|\partial_x^{n-j}\partial_y^{j+3} u\|_{L_2(0,T;L_{2,x_0/2,y_0/2}^{\rho'_j(x)})}\Bigr).
\end{multline}
Note that for $j=0$ the left part of this inequality is estimated in \eqref{2.19} (for $y_0=0$). Therefore, induction with respect to $j$ provides an appropriate estimate on
$\|\partial_x^{n-j}\partial_y^{j+3} u\|_{C\bigl([0,T];L_{2,x_0,y_0}^{\rho_j(x)}\bigr)}$
and $\|\partial_x^{n-j}\partial_y^{j+4} u\|_{L_2\bigl(0,T;L_{2,x_0,y_0}^{\rho'_j(x)}\bigr)}$. Finally, closure finishes the proof.
\end{proof}

\section{Well-posedness in regular classes}\label{S3}

First of all, we establish one bilinear estimate.

\begin{lemma}\label{L3.1}
Let the type b) or d) of boundary conditions \eqref{1.4} is considered. Let $u,v \in X^{k,\rho(x)}(\Pi_T^+)$ for certain $T>0$, natural $k\geq 4$, such that $k=3i$ or $k=3i+1$, $i\in\mathbb N$, and an admissible weight function $\rho(x)$, such that $\rho'(x)$ is also an admissible weight function and $\rho'(x)\geq 1 \ \forall x\geq 0$. Then $(uv)_x\in Y^{k,\rho(x)}(\Pi_T^+)$ and for any $t_0\in (0,T]$
\begin{multline}\label{3.1}
\|(uv)_x\|_{Y^{k,\rho(x)}(\Pi_{t_0}^+)} \\ \leq c(T)t_0^{1/2}\Bigl(\|u\|_{X^{k-2,\rho(x)}(\Pi_{t_0}^+)}\|v\|_{X^{k,\rho(x)}(\Pi_{t_0}^+)} + \|u\|_{X^{k,\rho(x)}(\Pi_{t_0}^+)}\|v\|_{X^{k-2,\rho(x)}(\Pi_{t_0}^+)}\Bigr) \\+
c\sum\limits_{3(m_1+m_2)\leq k-2} \|\partial_t^{m_1} u\big|_{t=0}\|_{H_+^{k-3m_1-2,\rho(x)}} \|\partial_t^{m_2} v\big|_{t=0}\|_{H_+^{k-3m_2-2,\rho(x)}}.
\end{multline}
\end{lemma}

\begin{proof}
Note that since
$$
\partial_y^j (uv) = \sum\limits_{l=0}^j {\binom jl} \partial_y^l u\partial_y^{j-l} v,
$$
in the case b) for odd values of $j<k-1$ either $\partial_y^l u\big|_{y=0}= \partial_y^l u\big|_{y=L}=0$ or $\partial_y^{j-l} v\big|_{y=0}= \partial_y^{j-l} v\big|_{y=L}=0$, therefore, $\partial_y^j (uv)_x\big|_{y=0}= \partial_y^j (uv)_x\big|_{y=L}=0$. In the case d) it is obvious that $\partial_y^j (uv)_x\big|_{y=0}= \partial_y^j (uv)_x\big|_{y=L}$ for $j<k-1$.

Let $k=3i$. In order to estimate $\partial_t^i (uv)$ in $L_2(0,t_0;L_{2,+}^{\rho^2(x)/\rho'(x)})$, consider $\partial_t^{m_1}u\partial_t^{m_2}v$, where $m_1+m_2=i$, $m_1\leq m_2$. Then since $\rho'(x)\geq 1$ with the use of \eqref{1.12}
\begin{multline}\label{3.2}
\|u\partial_t^iv\|_{L_2(0,t_0;L_{2,+}^{\rho^2(x)/\rho'(x)})} \leq 
\Bigl(\int_0^{t_0} \|u\rho^{1/2}\|^2_{L_{\infty,+}}
\|\partial_t^i v\|^2_{L_{2,+}^{\rho(x)}}\,dt\Bigr)^{1/2}  \\ \leq 
ct_0^{1/2} \|u\|_{C([0,t_0];H_+^{2,\rho(x)})}
\|\partial_t^i v\|_{C([0,t_0];L_{2,+}^{\rho(x)})} \leq
ct_0^{1/2} \|u\|_{X^{2,\rho(x)}(\Pi_{t_0}^+)}
\|v\|_{X^{k,\rho(x)}(\Pi_{t_0}^+)},
\end{multline}
and if $m_1\geq 1$ with the use of \eqref{1.10}
\begin{multline}\label{3.3}
\|\partial_t^{m_1}u \partial_t^{m_2}v\|_{L_2(0,t_0;L_{2,+}^{\rho^2(x)/\rho'(x)})} \leq  \Bigl(\int_0^{t_0} \|\partial_t^{m_1}u\rho^{1/2}\|^2_{L_{4,+}}
\|\partial_t^{m_2}v\rho^{1/2}\|^2_{L_{4,+}}\,dt\Bigr)^{1/2} \\ \leq
ct_0^{1/2} \|\partial_t^{m_1}u\|_{C([0,t_0];H_+^{1,\rho(x)})}
\|\partial_t^{m_2}v\|_{C([0,t_0];H_+^{1,\rho(x)})}  \\ \leq
ct_0^{1/2} \|u\|_{X^{k-2,\rho(x)}(\Pi_{t_0}^+)} 
\|v\|_{X^{k-2,\rho(x)}(\Pi_{t_0}^+)},
\end{multline}
since $3m_2+1\leq k-2$. 

Next, in order to estimate $\partial_t^m \partial_y^{k-3m-1}(uv)_x$ in $L_2(0,t_0;L_{2,+}^{\rho^2(x)/\rho'(x)})$, where $m\leq i-1$ if $k=3i$ and $m\leq i$ if $k=3i+1$, consider $\partial_t^{m_1}\partial^{\alpha_1}u \partial_t^{m_2}\partial^{\alpha_2}v$, where $3(m_1+m_2)+|\alpha_1|+|\alpha_2|=k$, $3m_1+|\alpha_1|\leq 3m_2+|\alpha_2|$. If $m_1=|\alpha_1|=0$ then similarly to \eqref{3.2}
$$
\|u\partial_t^{m_2}\partial^{\alpha_2}v\|_{L_2(0,t_0;L_{2,+}^{\rho^2(x)/\rho'(x)})} \leq ct_0^{1/2} \|u\|_{X^{2,\rho(x)}(\Pi_{t_0}^+)}
\|v\|_{X^{k,\rho(x)}(\Pi_{t_0}^+)},
$$
if $3m_2+|\alpha_2|\leq k-1$ then either $3m_1+|\alpha_1|\leq k-3$ or $k=4$, $m_1=m_2=0$, $|\alpha_1|=|\alpha_2|=2$. In the first case similarly to \eqref{3.3}
$$
\|\partial_t^{m_1}\partial^{\alpha_1}u \partial_t^{m_2}\partial^{\alpha_2}v\|_{L_2(0,t_0;L_{2,+}^{\rho^2(x)/\rho'(x)})} \leq ct_0^{1/2} \|u\|_{X^{k-2,\rho(x)}(\Pi_{t_0}^+)} \|v\|_{X^{k,\rho(x)}(\Pi_{t_0}^+)}.
$$
In the second case similarly to \eqref{3.2}
\begin{multline*}
\|\partial^{\alpha_1}u\partial^{\alpha_2}v\|_{L_2(0,t_0;L_{2,+}^{\rho^2(x)/\rho'(x)})} \leq 
\Bigl(\int_0^{t_0} \|\partial^{\alpha_1} u\|^2_{L_{2,+}^{\rho(x)}}
\|\partial^{\alpha_2}v\rho^{1/2}\|^2_{L_{\infty,+}} \,dt\Bigr)^{1/2}  \\ \leq 
ct_0^{1/2} \|u\|_{C([0,t_0];H_+^{2,\rho(x)})}
\|\partial^{\alpha_2} v\|_{C([0,t_0];H_+^{2,\rho(x)})} \leq
ct_0^{1/2} \|u\|_{X^{2,\rho(x)}(\Pi_{t_0}^+)}
\|v\|_{X^{4,\rho(x)}(\Pi_{t_0}^+)}.
\end{multline*}

In order to estimate $\partial_t^m (uv)_x$ in $C([0,t_0];H_+^{k-3m-3,\rho(x)})$ for $3m+3\leq k$, evaluate first $\partial_t^{m+1} (uv)_x$ in $L_1(0,t_0;H_+^{k-3m-3,\rho(x)})$. To this end consider $\partial_t^{m_1}\partial^{\alpha_1}u \partial_t^{m_2}\partial^{\alpha_2}v$, where $m_1+m_2=m+1$, $|\alpha_1|+|\alpha_2|\leq k-3m-2$, $3m_1+|\alpha_1|\leq 3m_2+|\alpha_2|$. If $m_1=|\alpha_1|=0$ then again  
since $\rho'(x)\geq 1$ with the use of \eqref{1.12}
\begin{multline*}
\|u\partial_t^{m_2}\partial^{\alpha_2}v\|_{L_1(0,t_0;L_{2,+}^{\rho(x)})} \leq
\|u\rho^{1/2}\|_{L_2(0,t_0;L_{\infty,+})} 
\|\partial_t^{m_2}\partial^{\alpha_2}v\|_{L_2(0,t_0;L_{2,+}^{\rho'(x)})} \\ \leq
ct_0^{1/2} \|u\|_{C([0,t_0];H_+^{2,\rho(x)})}
\|\partial_t^{m_2}v\|_{L_2(0,t_0;H_+^{k-3m_2+1,\rho'(x)})}  \\ \leq
ct_0^{1/2} \|u\|_{X^{2,\rho(x)}(\Pi_{t_0}^+)}
\|v\|_{X^{k,\rho(x)}(\Pi_{t_0}^+)}.
\end{multline*}
If $3m_2+|\alpha_2|\leq k$ then either $3m_1+|\alpha_1|\leq k-3$ or $k=4$, $m_1=0$, $m_2=1$, $|\alpha_1|=2$, $|\alpha_2|=0$. In the first case similarly to \eqref{3.3}
\begin{multline*}
\|\partial_t^{m_1}\partial^{\alpha_1}u \partial_t^{m_2}\partial^{\alpha_2}v\|_{L_1(0,t_0;L_{2,+}^{\rho(x)})} \\  \leq
\|\partial_t^{m_1}\partial^{\alpha_1}u (\rho')^{1/2}\|_{L_2(0,t_0;L_{\infty,+})}
\|\partial_t^{m_2}\partial^{\alpha_2}v\|_{L_2(0,t_0;L_{2,+}^{\rho(x)})} \\ \leq
ct_0^{1/2} \|\partial_t^{m_1} u\|_{L_2(0,t_0;H_+^{k-3m_1-1,\rho'(x)})}
\|\partial_t^{m_2}v\|_{C([0,t_0];H_+^{k-3m_2,\rho(x)})} \\ \leq
ct_0^{1/2} \|u\|_{X^{k-2,\rho(x)}(\Pi_{t_0}^+)}
\|v\|_{X^{k,\rho(x)}(\Pi_{t_0}^+)}.
\end{multline*}
In the second case
\begin{multline*}
\|\partial^{\alpha_1}u v_t\||_{L_1(0,t_0;L_{2,+}^{\rho(x)})}   \leq
\Bigl(\int_0^{t_0} \|\partial^{\alpha_1} u\|_{L_{2,+}^{\rho(x)}} \|v_t(\rho')^{1/2}\|_{L_{\infty,+}}\,dt\Bigr)^{1/2} \\ \leq ct_0^{1/2} \|u\|_{C([0,t_0];H_+^{2,\rho(x)})} \|v_t\|_{L_2(0,t_0;H_+^{2,\rho'(x)})} \leq 
ct_0^{1/2}\|u\|_{X^{2,\rho(x)}(\Pi_{t_0}^+)}
\|v\|_{X^{4,\rho(x)}(\Pi_{t_0}^+)}.
\end{multline*}
Now evaluate $\partial_t^m(uv)_x\big|_{t=0}$ in $H_+^{k-3m-3,\rho(x)}$. To this end consider $(\partial_t^{m_1}\partial^{\alpha_1}u \partial_t^{m_2}\partial^{\alpha_2}v)\big|_{t=0}$, where $3(m_1+m_2)+|\alpha_1|+|\alpha_2|\leq k-2$, $3m_1+|\alpha_1|\leq 3m_2+|\alpha_2|$. If $m_1=|\alpha_1|=0$ then
\begin{multline*}
\|(u\partial_t^{m_2}\partial^{\alpha_2}v)\big|_{t=0}\|_{L_{2,+}^{\rho(x)}} \leq	
\|u\big|_{t=0}\|_{L_{\infty,+}} 
\|\partial_t^{m_2}\partial^{\alpha_2}v\big|_{t=0}\|_{L_{2,+}^{\rho(x)}} \\ \leq
\|u\big|_{t=0}\|_{H_+^{2,\rho(x)}}
c\|\partial_t^{m_2}v\big|_{t=0}\|_{H_+^{k-3m_2-2,\rho(x)}}.
\end{multline*}
If $3m_2+|\alpha_2|\leq k-3$ then
\begin{multline*}
\|(\partial_t^{m_1}\partial^{\alpha_1}u \partial_t^{m_2}\partial^{\alpha_2}v)\big|_{t=0}\|_{L_{2,+}^{\rho(x)}} \leq c
\|\partial_t^{m_1}\partial^{\alpha_1}u\big|_{t=0}\rho^{1/2}\|_{L_{4,+}}
\|\partial_t^{m_2}\partial^{\alpha_2}v\big|_{t=0}\rho^{1/2}\|_{L_{4,+}} \\ \leq c
\|\partial_t^{m_1}u\big|_{t=0}\|_{H_+^{k-3m_1-2,\rho(x)}}
\|\partial_t^{m_2}v\big|_{t=0}\|_{H_+^{k-3m_2-2,\rho(x)}}.
\end{multline*}

Finally, appropriate estimates on $\partial_t^m (uv)_x$ in $L_2(0,t_0;H_+^{k-3m-2,\rho'(x)})$ for $3m+3\leq k$ can be obtained quite similarly to \eqref{3.2}, \eqref{3.3}.
\end{proof}

Now we can prove the main result of this section.

\begin{proof}[Proof of Theorem~\ref{T1.2}]
 In order to set to zero boundary data at $x=0$ we use special functions $J(t,x,y;\mu)$ of "boundary potential" type, constructed in \cite{F18-1} (without lost of generality we assume that $\mu\in \widetilde H^{(k+1)/3,k+1}(B)$). We do not intend to repeat here the definition of these functions but only describe their main properties, proved in \cite{F18-1}.

Any function $J$ is infinitely smooth for $x>0$ and satisfy equality \eqref{2.1} for $f\equiv 0$. For any $T>0$, $x_0>0$, $n$ and $j$
\begin{equation}\label{3.4}
\sup\limits_{x\geq x_0} \|J(\cdot,x,\cdot;\mu\|_{\widetilde H^{j,3j}(B_T)} \leq 
c(T,x_0,n,j,k)\|\mu\|_{\widetilde H^{(k+1)/3,k+1}(B)}.
\end{equation}
Next, for $m\leq k/3$ 
$$ 
\|\partial_t^m J\|_{C_b(\mathbb R^t;\widetilde H_+^{k-3m})} \leq c(k) \|\mu\|_{\widetilde H^{(k+1)/3,k+1}(B)}
$$
and if $3m+|\alpha|\leq k+1$
$$
\|\partial_t^m \partial^\alpha J\|_{C_b(\overline{\mathbb R}_+^x;L_2(B))} \leq c(k) \|\mu\|_{\widetilde H^{(k+1)/3,k+1}(B)},
$$
moreover, $J(t,0,y;\mu)\equiv \mu(t,y)$.

Let
\begin{equation}\label{3.5}
\psi(t,x,y)\equiv J(t,x,y;\mu)\eta(2-x).
\end{equation}
The aforementioned properties of the function $J$ provide that $\psi\in X^{k,\rho(x)}(\Pi_T^+)$ for any admissible function $\rho$, moreover, $\psi\big|_{x=0}=\mu$ and
\begin{equation}\label{3.6}
\|\psi\|_{X^{k,\rho(x)}(\Pi_T^+)} \leq c(\rho,k)\|\mu\|_{\widetilde H^{(k+1)/3,k+1}(B)}.
\end{equation}
Let
\begin{equation}\label{3.7}
\widetilde\psi\equiv \psi_t +b\psi_x +\psi_{xxx} +\psi_{xyy}.
\end{equation}
It is easy to see that
$$
\widetilde\psi =-bJ\eta'(2-x) -J\eta'''(2-x) +3J_x\eta''(2-x) -3J_{xx}\eta'(2-x) -J_{yy}\eta'(2-x),
$$
therefore, $\widetilde\psi\in C^\infty(\overline{\Pi}_T^+)$, $\widetilde\psi =0$ if $x\in [0,1]\cup[2,+\infty)$. In particular, $\widetilde\psi \in Y^{k,\rho(x)}(\Pi_T^+)$ and
\begin{equation}\label{3.8}
\|\widetilde\psi\|_{Y^{k,\rho(x)}(\Pi_T^+)} \leq c(T) \|\mu\|_{\widetilde H^{(k+1)/3,k+1}(B)}.
\end{equation}

Let
\begin{gather}\label{3.9}
U(t,x,y) \equiv u(t,x,y) -\psi(t,x,y), \\
\label{3.10}
U_0\equiv u_0 -\psi\big|_{t=0}, \qquad F\equiv -\widetilde\psi -\psi\psi_x.
\end{gather}
Instead of \eqref{1.1}--\eqref{1.4} consider in $\Pi_T^+$ an initial-boundary value problem for an equation
\begin{equation}\label{3.11}
U_t+bU_x+U_{xxx}+U_{xyy}+UU_x+(\psi U)_x=F,
\end{equation}
with initial and boundary conditions
\begin{equation}\label{3.12}
U\big|_{t=0} = U_0,\qquad U\big|_{x=0}=0
\end{equation}
and the same boundary conditions on $\Omega_{T,+}$ as \eqref{1.4}. 
Note that $U_0\in \widetilde H^{k,\rho(x)}_+$, $F\in Y^{k,\rho(x)}(\Pi_T^+)$. Define also for this problem special functions by analogy with $\Phi_m$: let $\Phi^*_0\equiv U_0$ and for $m\geq 1$
\begin{multline*}
\Phi^*_m\equiv \partial_t^{m-1} F\big|_{t=0} - (\partial_x^3 +\partial_x\partial_y^2 +b\partial_x)\Phi^*_{m-1} \\ -
\sum\limits_{l=0}^{m-1} {\binom {m-1}l} \Bigl[\Phi^*_l\partial_x\Phi^*_{m-l-1} +
\bigl(\partial_t^l\psi\big|_{t=0} \Phi^*_{m-l-1}\bigr)_x\Bigr].
\end{multline*}
It is easy to see, that $\Phi_m^*=\Phi_m -\partial_t^m\psi\big|_{t=0}$.

For $t_0\in (0,T]$ consider a set of functions
$$
\widetilde X^{k,\rho(x)}(\Pi_{t_0}^+) =\bigl\{v\in X^{k,\rho(x)}(\Pi_{t_0}^+):\ 
\partial_t^m v\big|_{t=0} =\Phi^*_m\ \mbox{for } m<k/3\bigr\}.
$$
Define on this set a map $u=\Lambda v$, where $u \in \widetilde X^{k,\rho(x)}(\Pi_{t_0}^+)$ is a solution to linear problem 
\begin{equation}\label{3.13}
u_t+bu_x+u_{xxx}+u_{xyy} = f\equiv F-(\psi v)_x -vv_x,
\end{equation}
with initial and boundary conditions $u\big|_{t=0}=U_0$, $u\big|_{x=0}=0$ and \eqref{1.4}. Note that by virtue of Lemma~\ref{L3.1} $f\in Y^{k,\rho(x)}(\Pi_{t_0}^+)$. It easy to see that the corresponding functions $\widetilde\Phi_m$, written for this problem in accordance with Definition~\ref{D2.2} for $m<k/3$, coincide with $\Phi^*_m$, therefore, $\widetilde\Phi_m\big|_{x=0}=0$. Then Lemma~\ref{L2.3} provides that the map $\Lambda$ exists. Moreover, inequalities \eqref{2.9}, \eqref{3.1}, \eqref{3.6} and \eqref{3.8} yield that
\begin{multline}\label{3.14}
\|\Lambda v\|_{X^{k,\rho(x)}(\Pi_{t_0}^+)} \leq c(T)\Bigl(\|u_0\|_{H_+^{k,\rho(x)}}
+ \|u_0\|^2_{H_+^{k,\rho(x)}} + \|\mu\|_{H^{(k+1)/3,k+1}(B_T)}  \\+ \|\mu\|^2_{H^{(k+1)/3,k+1}(B_T)}  +
t_0^{1/2}\|\mu\|_{H^{(k+1)/3,k+1}(B_T)}\|v\|_{X^{k,\rho(x)}(\Pi_{t_0}^+)}  \\+
t_0^{1/2}\|v\|_{X^{k-2,\rho(x)}(\Pi_{t_0}^+)}\|v\|_{X^{k,\rho(x)}(\Pi_{t_0}^+)}
\Bigr).
\end{multline}
Similarly,
\begin{multline}\label{3.15}
\|\Lambda v_1 -\Lambda v_2\| _{X^{k,\rho(x)}(\Pi_{t_0}^+)} \leq c(T) t_0^{1/2}\Bigl(
\|\mu\|_{H^{(k+1)/3,k+1}(B_T)}\|v_1-v_2\|_{X^{k,\rho(x)}(\Pi_{t_0}^+)} \\+
\bigl(\|v_1\|_{X^{k,\rho(x)}(\Pi_{t_0}^+)} +\|v_2\|_{X^{k,\rho(x)}(\Pi_{t_0}^+)}\bigr)
\|v_1-v_2\|_{X^{k,\rho(x)}(\Pi_{t_0}^+)}\Bigr).
\end{multline}
Therefore, existence of the unique fixed point $U$ of the map $\Lambda$ for certain $t_0$, depending on $\|u_0\|_{H_+^{k,\rho(x)}}$ and $\|\mu\|_{H^{(k+1)/3,k+1}(B_T)}$, follows by the standard argument. Then $u\equiv U+\psi\in X^{k,\rho(x)}(\Pi_{t_0}^+)$ is the unique solution to the original problem.

Note that Theorem~\ref{T1.1} provides that $U\in X^{3,\rho(x)}(\Pi_T^+)$, then application of inequality \eqref{3.14} to the function $U$ and induction with respect to $k$ imply that $u\in X^{k,\rho(x)}(\Pi_T^+)$.

Finally, we prove continuous dependence. Let $M>0$, let the functions $u_{0j}$, $\mu_j$,  satisfy the hypothesis of Theorem~\ref{T1.2} and $\|(u_{0j},\mu_j)\|_{\widetilde H_+^{k,\rho(x)} \times \widetilde H^{(k+1)/3,k+1}(B_T)}\leq M$ for $j=1$ and $2$, then for the corresponding solutions $\|u_j\|_{X^{k,\rho(x)}(\Pi_T^+)}\leq c_0(M)$. Define the functions $\psi_j$ and $U_j$ by the corresponding analogs of formulas \eqref{3.5} and \eqref{3.9}. Then similarly to \eqref{3.15} for $t_0\in (0,T]$
\begin{multline*}
\|U_1- U_2\|_{X^{k,\rho(x)}(\Pi_{t_0}^+)} \leq c(M)\Bigl(
\|u_{01}-u_{02}\|_{\widetilde H_+^{k,\rho(x)}} +
\|\mu_1-\mu_2\|_{\widetilde H^{(k+1)/3,k+1}(B_T)} \\+
t_0^{1/2}\|U_1- U_2\|_{X^{k,\rho(x)}(\Pi_{t_0}^+)}\Bigr),
\end{multline*}
whence the desired result immediately succeeds.
\end{proof}

\begin{remark}\label{R3.1}
The reason why the implemented scheme does not work for the types of boundary conditions a) and c) is that in the linear case solutions are constructed as limits of smooth ones where the condition $f_{yy}\big|_{y=0}=0$ is necessary. Then this condition is inherited in the spaces $Y^{k,\rho(x)}$ for $k\geq 4$. For ZK equation itself it means that $u_y\big|_{y=0}=0$ since here $f_{yy}=-(uu_{xyy}+2u_yu_{xy}+u_xu_{yy})$, but such boundary condition is superfluous.
\end{remark}

\section{Internal regularity}\label{S4}

Establish, first, one auxiliary lemma, which is similar to Lemma~\ref{L3.1}, but is valid for all types of boundary conditions.
In fact, in an implicit form it was established in \cite{F18-1}.

\begin{lemma}\label{L4.1}
Let $u,v \in X^{3,\rho(x)}(\Pi_T^+)$ for certain $T>0$ and an admissible weight function $\rho(x)$, such that $\rho'(x)$ is also an admissible weight function and $\rho'(x)\geq 1 \ \forall x\geq 0$. Then $(uv)_x\in Y^{3,\rho(x)}(\Pi_T^+)$ and for any $t_0\in (0,T]$
\begin{multline}\label{4.1}
\|(uv)_x\|_{Y^{3,\rho(x)}(\Pi_{t_0}^+)} \\ \leq c(T)t_0^{1/2}\Bigl(\|u\|_{X^{2,\rho(x)}(\Pi_{t_0}^+)}\|v\|_{X^{3,\rho(x)}(\Pi_{t_0}^+)} + \|u\|_{X^{3,\rho(x)}(\Pi_{t_0}^+)}\|v\|_{X^{2,\rho(x)}(\Pi_{t_0}^+)}\Bigr) \\+
c\|u\big|_{t=0}\|_{H_+^{2,\rho(x)}} \|v\big|_{t=0}\|_{H_+^{2,\rho(x)}}.
\end{multline}
\end{lemma}

\begin{proof}
The type of boundary conditions here is irrelevant, since $(uv)\big|_{y=0}= (uv)\big|_{y=L}=0$ in the case a), $(uv)_y\big|_{y=0}=(uv)_y\big|_{y=L}=0$ in the case b), $(uv)\big|_{y=0}=(uv)_y\big|_{y=L}= 0$ in the case c), $(uv)\big|_{y=0}= (uv)\big|_{y=L}$, $(uv)_y\big|_{y=0}=(uv)_y\big|_{y=L}$ in the case d)  and it is sufficient for the following argument.

Estimate \eqref{3.2} remains the same, estimate \eqref{3.3} is omitted. 

In order to estimate $\partial_y^{2}(uv)_x$ in $L_2(0,t_0;L_{2,+}^{\rho^2(x)/\rho'(x)})$, consider $\partial^{\alpha_1}u \partial^{\alpha_2}v$, where $|\alpha_1|+|\alpha_2|=3$, $|\alpha_1|\leq |\alpha_2|$. If $|\alpha_1|=0$ then similarly to \eqref{3.2}
$$
\|u\partial^{\alpha_2}v\|_{L_2(0,t_0;L_{2,+}^{\rho^2(x)/\rho'(x)})} \leq ct_0^{1/2} \|u\|_{X^{2,\rho(x)}(\Pi_{t_0}^+)}
\|v\|_{X^{3,\rho(x)}(\Pi_{t_0}^+)},
$$
if $|\alpha_2|= 2$ then $|\alpha_1|= 1$ and similarly to \eqref{3.3}
$$
\|\partial^{\alpha_1}u \partial^{\alpha_2}v\|_{L_2(0,t_0;L_{2,+}^{\rho^2(x)/\rho'(x)})} \leq ct_0^{1/2} \|u\|_{X^{2,\rho(x)}(\Pi_{t_0}^+)} \|v\|_{X^{3,\rho(x)}(\Pi_{t_0}^+)}.
$$

In order to estimate $(uv)_x$ in $C([0,t_0];L_{2,+}^{\rho(x)})$, evaluate first $\partial_t (uv)_x$ in $L_1(0,t_0;L_{2,+}^{\rho(x)})$. To this end, consider $\partial^{\alpha_1}u \partial^{\alpha_2}v_t$, where $|\alpha_1|+|\alpha_2|\leq 1$. If $|\alpha_1|=0$ then since $\rho'(x)\geq 1$ with the use of \eqref{1.12}
\begin{multline*}
\|u\partial^{\alpha_2}v_t\|_{L_1(0,t_0;L_{2,+}^{\rho(x)})} \leq
t_0^{1/2}\|u\rho^{1/2}\|_{C([0,t_0];L_{\infty,+})} 
\|\partial^{\alpha_2}v_t\|_{L_2(0,t_0;L_{2,+}^{\rho'(x)})} \\ \leq
ct_0^{1/2} \|u\|_{C([0,t_0];H_+^{2,\rho(x)})}
\|v_t\|_{L_2(0,t_0;H_+^{1,\rho'(x)})}   \leq
ct_0^{1/2} \|u\|_{X^{2,\rho(x)}(\Pi_{t_0}^+)}
\|v\|_{X^{3,\rho(x)}(\Pi_{t_0}^+)},
\end{multline*}
and if $|\alpha_2|=0$ then
\begin{multline*}
\|\partial^{\alpha_1}u v_t\|_{L_1(0,t_0;L_{2,+}^{\rho(x)})}   \leq
t_0^{1/2} \|\partial^{\alpha_1} u\rho^{1/2}\|_{C([0,t_0];L_{4,+})} \|v_t (\rho')^{1/2}\|_{L_2(0,t_0;L_{4,+})} \\ \leq
ct_0^{1/2} \|u\|_{C([0,t_0];H_+^{2,\rho(x)})}
\|v_t\|_{L_2(0,t_0;H_+^{1,\rho'(x)})}   \leq
ct_0^{1/2} \|u\|_{X^{2,\rho(x)}(\Pi_{t_0}^+)}
\|v\|_{X^{3,\rho(x)}(\Pi_{t_0}^+)}.
\end{multline*}
It is also obvious, that
$$
\|(uv)_x\big|_{t=0}\|_{L_{2,+}^{\rho(x)}} \leq c\|u\big|_{t=0}\|_{H_+^{2,\rho(x)}} \|v\big|_{t=0}\|_{H_+^{2,\rho(x)}}.
$$

Finally, it is easy to see that
$$
\|(uv)_x\|_{L_2(0,t_0;H_+^{1,\rho'(x)})} \leq ct_0\|u\|_{X^{2,\rho(x)}(\Pi_{t_0}^+)}
\|v\|_{X^{2,\rho(x)}(\Pi_{t_0}^+)}.
$$
\end{proof}

\begin{proof}[Proof of Theorem~\ref{T1.3}]
We use induction with respect to $n$. For $n=0$ the result, of course, is a simple corollary of Theorem~\ref{T1.1}. Let $n\geq 1$ and $\partial_x^{n-1} u\in X^{3,\rho(x)}(\Pi_T^{x_0})$ for any $x_0>0$.

Introduce the function $U$ by formula \eqref{3.9}, where $u\in X^{3,\rho(x)}(\Pi_T^+)$ is the solution to problem \eqref{1.1}--\eqref{1.4} from Theorem~\ref{T1.1}. Then $U\in X^{3,\rho(x)}(\Pi_T^+)$ can be considered as a solution to a linear initial-boundary value problem in $\Pi_T^+$ to an equation
\begin{equation}\label{4.2}
U_t+bU_x +U_{xxx}+U_{xyy}=f\equiv uu_x-\widetilde\psi 
\end{equation}
with initial and boundary conditions \eqref{3.10}, \eqref{1.4} (the function $\widetilde\psi$ is defined in \eqref{3.7}). Property \eqref{3.4} of the boundary potential implies that $\partial_x^l\psi\in X^{3,\rho(x)}(\Pi_T^{x_0})$ for all $x_0>0$ and natural $l$. Lemma~\ref{L4.1} implies that $uu_x\in Y^{3,\rho(x)}(\Pi_T^+)$. Taking into account also \eqref{3.8}, one can see that the hypothesis of Lemma~\ref{L2.3} is verified for the function $U$ in the case $k=3$. 
It is obvious that $\partial_x^l\widetilde\psi\in C([0,T];\widetilde H_{x_0}^{2,\rho^2(x)/\rho'(x)})$ for any $x_0>0$ and natural $l$.

Next, we show that
\begin{equation}\label{4.3}
\partial_x^n (uu_x) \in L_2(0,T; \widetilde H_{x_0}^{2,\rho^2(x)/\rho'(x)}).
\end{equation}
To this end, consider $\partial^{\alpha_1}u \partial^{\alpha_2}u$, where $\alpha_1=(i_1,j_1)$, $\alpha_2=(i_2,j_2)$, $|\alpha_1|+|\alpha_2|=n+3$, $|\alpha_1|\leq |\alpha_2|$, $j_1+j_2\leq 2$. If $|\alpha_1|=0$, then
$$
\iiint_{\Pi_T^{x_0}} ( u \partial^{\alpha_2}u)^2\frac{\rho^2}{\rho'}\,dxdydt   \leq
\sup\limits_{\Pi_T^+} \Bigl(u \frac{\rho}{\rho'}\Bigr)^2 
\iiint_{\Pi_T^{x_0}} (\partial^{\alpha_2}u)^2\rho'\,dxdydt <\infty,
$$
since $\rho/\rho'\leq c\rho^{1/2}$, $u\in C([0,T];H_+^{3,\rho(x)})$ and, therefore,
$u \rho^{1/2}\in C_b(\overline{\Pi}_T^+)$, $\partial^{\alpha_2} u \in L_2(0,T;L_{2,x_0}^{\rho'(x)})$ by virtue of the inductive hypothesis. 

If $|\alpha_2|\leq n+2$ then for any $t\in [0,T]$ with the use of \eqref{1.10} 
$$
\Bigl(\iint_{\Sigma_{x_0}} (\partial^{\alpha_2}u)^4\rho^2\,dxdy\Bigr)^{1/2}   \leq \sum\limits_{l=0}^{n-1}
\|\partial_x^{l} u\|^2_{C([0,T];H_{x_0}^{3,\rho(x)})} <\infty,
$$
and, therefore, $\partial^{\alpha_1}u \partial^{\alpha_2}u \in L_2(0,T;L_{2,x_0}^{\rho^2(x)/\rho'(x)})$.

Once \eqref{4.3} is established, it is suffice to apply Lemma~\ref{L2.5}.
\end{proof} 

\begin{proof}[Proof of Theorem~\ref{T1.4}]
We start with the following auxiliary assertion. Let for certain $x_0>0$, $y_0\in (0,L/2)$ and an admissible weight function $\rho(x)$, such that $\rho'$ is also admissible,
\begin{equation}\label{4.4}
v\in C([0,T];L_{2,x_0,y_0}^{\rho(x)}),\quad
|Dv| \in L_2(0,T;_{L_2,x_0,y_0}^{\rho'(x)}).
\end{equation}  
Then by virtue of \eqref{1.10}
\begin{equation}\label{4.5}
\iiint_{\Pi_T^{x_0,y_0}} v^4 \rho'\rho\,dxdydt <\infty.
\end{equation}

Now as in the proof of Theorem~\ref{T1.3} consider the function $U$ as the solution to the linear initial-boundary value problem for equation \eqref{4.1}. Assume that for certain $1\leq j\leq n$, any $x_0>0$, $y_0\in (0,L/2)$ and any multi-index $\alpha=(l_1,l_2)$, such that $|\alpha|\leq n+3$, $l_2\leq j+2$, for $v\equiv \partial^\alpha u$ and $\rho\equiv \rho_{j-1}$ property \eqref{4.4} is verified. Then \eqref{4.5} implies that since $\rho'_{j-1}(x)\geq c\rho_j(x)>0$ and, therefore, $\rho_{j-1}(x)\geq c>0$
$$
\partial_x^{m-j}\partial_y^{j+2}(uu_x)\in L_2(0,T;L_{2,x_0,y_0}^{\rho'_{j-1}(x)})
$$
if $j\leq m\leq n$. Then it follows from Lemma~\ref{L2.6} that
$$
\partial_x^{n-j}\partial_y^{j+3} U \in C\bigl([0,T];L_{2,x_0,y_0}^{\rho_j(x)}\bigr),\quad
\partial_x^{n-j}\partial_y^{j+4} U \in L_2\bigl(0,T;L_{2,x_0,y_0}^{\rho'_j(x)}\bigr).
$$
Therefore, for any multi-index $\alpha=(l_1,l_2)$, such that $|\alpha|\leq n+3$, $l_2\leq j+3$,
$$
\partial^\alpha u\in C([0,T];L_{2,x_0,y_0}^{\rho_{j}(x)}),\quad
|D\partial^\alpha u| \in L_2(0,T;_{L_2,x_0,y_0}^{\rho'_{j}(x)}).
$$
Since for $j=1$, $v\equiv \partial^\alpha u$ property \eqref{4.4} follows from Theorem~\ref{T1.3}, induction with respect to $j$ provides the desired result.
\end{proof}

\section{Large-time decay of small solutions}\label{S5}

\begin{proof}[Proof of Theorem~\ref{T1.4}]
Let $\alpha>0$, $\rho(x)\equiv \frac1{2\alpha}e^{2\alpha x}$, $u_0\in \widetilde H_+^{3,\rho(x)}$, $u_0(0,y)\equiv 0$, $\mu\equiv 0$. Consider the solution to problem \eqref{1.1}--\eqref{1.4} (in the cases a) and c)) $u\in X^{3,\rho(x)}(\Pi_T^+)\ \forall T$. 

Multiplication of equation \eqref{1.1} by $2u(t,x,y)$ and consequent integration obviously provides an inequality
\begin{equation}\label{5.1}
\|u(t,\cdot,\cdot)\|_{L_{2,+}}\leq \|u_0\|_{L_{2,+}} \quad \forall t\geq 0,
\end{equation} 
which is, of course, the analog of the conservation law for Zakharov--Kuznetsov equation. Multiplication of \eqref{1.1} by $2u(t,x,y)\rho(x)$ provides an equality
\begin{multline}\label{5.2}
\frac{d}{dt}\iint u^2\rho\,dxdy  + \int_0^L u_x^2\big|_{x=0}\,dy+
2\alpha \iint (3u_x^2+u_y^2)\rho\,dxdy \\
-2\alpha(b+4\alpha^2) \iint u^2\rho\,dxdy
= \frac{2\alpha}3 \iint u^3\rho\,dxdy.
\end{multline} 
With the use of \eqref{1.10} and \eqref{5.1} one can easily show that uniformly with respect to $L$
\begin{equation}\label{5.3}
\frac{2}3 \iint u^3\rho\,dxdy \leq \frac12 \iint |Du|^2\rho\,dxdy  +
c (\|u_0\|_{L_{2,+}}+\|u_0\|_{L_{2,+}}^2)\iint u^2\rho\,dxdy
\end{equation}
(for more details see \cite{F18-1}).
Inequality \eqref{1.13} yields that for certain constant $c_0$
\begin{equation}\label{5.4}
\frac12 \iint u_y^2\rho\,dxdy \geq \frac{c_0}{L^2} \iint u^2\rho\,dxdy.
\end{equation}
Combining \eqref{5.2}--\eqref{5.4} we find that uniformly with respect to $\alpha$ and $L$
\begin{multline}\label{5.5}
\frac{d}{dt}\iint u^2\rho\,dxdy + \int_0^L u_x^2\big|_{x=0}\,dy +
\alpha \iint |Du|^2\rho\,dxdy \\
+\alpha\Bigl(\frac{c_0}{L^2}-2b-8\alpha^2-c (\|u_0\|_{L_{2,+}}+\|u_0\|_{L_{2,+}}^2)\Bigr) \iint u^2\rho\,dxdy
\leq 0.
\end{multline}
Choose $\displaystyle L_0= \frac12\sqrt{\frac{c_0}b}$ if $ b>0$, $\displaystyle \alpha_0 = \frac{\sqrt{c_0}}{8L}$, $\epsilon_0>0$ satisfying an inequality $\displaystyle \epsilon_0+\epsilon_0^2 \leq \frac {c_0}{8cL^2}$, $\beta = \displaystyle \frac{c_0}{4L^2}$. Then it follows from \eqref{5.5} that for $\alpha\in (0,\alpha_0]$ and $\|u_0\|_{L_{2,+}}\leq \epsilon_0$
\begin{equation}\label{5.6}
\frac{d}{dt}\iint u^2\rho\,dxdy  
+\alpha\beta\iint u^2\rho\,dxdy
\leq 0,
\end{equation}
which, in turn, yields that
\begin{equation}\label{5.7}
\|e^{\alpha x}u(t,\cdot,\cdot)\|^2_{L_{2,+}} \leq e^{-\alpha\beta t} \|e^{\alpha x}u_0\|^2_{L_{2,+}}\qquad \forall t\geq 0.
\end{equation}
Moreover, it was shown in \cite{F18-1} that under the same assumptions on $\alpha$ and $\|u_0\|_{L_{2,+}}$
\begin{equation}\label{5.8}
e^{\alpha\beta t}\|u(t,\cdot,\cdot)\|^2_{H_+^{1,\exp(2\alpha x)}} + \int_0^t e^{\alpha\beta \tau} \iint |D^2u|^2\rho\, dxdyd\tau \leq c\qquad \forall t\geq 0,
\end{equation}
where the constant $c$ depends on $b$, $\alpha$, $\beta$, $\|u_0\|_{H_+^{1,\exp(2\alpha x)}}$.

Next, note that Lemma~\ref{L4.1} implies that for $f\equiv -uu_x$ the hypothesis of Lemma~\ref{L2.5} is verified for all $T>0$. Apply equality \eqref{2.16}, then for a.e. $t>0$
\begin{multline}\label{5.9}
\frac{d}{dt} \iint u_t^2\rho\,dxdy + \rho(0)\int_0^L \nu_1^2\,dy +
\iint (3u_{tx}^2+u_{ty}^2-bu_t^2)\rho'\,dxdy \\ -
\iint u_t^2\rho'''\,dxdy = 2\iint uu_tu_{tx}\rho\,dxdy +4\alpha\iint uu^2_t\rho\,dxdy.
\end{multline}
Here inequality \eqref{1.10} implies that
\begin{multline*}
\iint uu_tu_{tx}\rho\,dxdy \leq \Bigl(\iint u^4\,dxdy\Bigr)^{1/4} \Bigl(\iint u_t^4\rho^2\,dxdy\Bigr)^{1/4} \Bigl(\iint u^2_{tx}\rho\,dxdy\Bigr)^{1/2} \\ \leq
\varepsilon \iint |Du_t|^2\rho\,dxdy +c(\varepsilon)\bigl(\|u\|^4_{H^1_+} +\|u\|_{H_+^1}\bigr)\iint u_t^2\rho\,dxdy,
\end{multline*}
where $\varepsilon>0$ can be chosen arbitrarily small. The second term in the right side  of \eqref{5.9} is estimated in a similar way. As a result, equality \eqref{5.9} yields that similarly to \eqref{5.5}
\begin{multline}\label{5.10}
\frac{d}{dt}\iint u_t^2\rho\,dxdy + \alpha \iint |Du_t|^2\rho\,dxdy \\
+\alpha\Bigl(\frac{c_0}{L^2}-2b-8\alpha^2-c(\alpha) \bigl(\|u\|^4_{H^1_+} +\|u\|_{H_+^1}\bigr)\Bigr) \iint u_t^2\rho\,dxdy \leq 0,
\end{multline}
where $c_0$ is the same constant as in \eqref{5.4}. According to \eqref{5.8} choose $T_0>0$ such that
$$
c(\alpha) \bigl(\|u\|^4_{H^1_+} +\|u\|_{H_+^1}\bigr) \leq \frac{c_0}{8L^2}\qquad \forall t\geq T_0.
$$
Then similarly to \eqref{5.6} we derive that for $t\geq T_0$
$$
\frac{d}{dt}\iint u_t^2\rho\,dxdy  
+\alpha\beta\iint u_t^2\rho\,dxdy
\leq 0,
$$
and, therefore,
\begin{equation}\label{5.11}
\|e^{\alpha x}u_t(t,\cdot,\cdot)\|^2_{L_{2,+}} \leq ce^{-\alpha\beta t}\qquad \forall t\geq 0,
\end{equation}
where the constant $c$ depends on $b$, $\alpha$, $\beta$, $\|u_0\|_{H_+^{3,\exp(2\alpha x)}}$.

Next, write down equality \eqref{2.17}:
\begin{multline}\label{5.12}
\frac{d}{dt} \iint (u_{xy}^2+u_{yy}^2)\rho\,dxdy +
\int_0^L (u_{xxy}^2\rho +2u_{xxy}u_{xy}\rho' -u_{xy}^2\rho'' +bu_{xy}^2\rho)\big|_{x=0}\,dy \\ + \iint (3u_{xxy}^2 +4u_{xyy}^2 +u_{yyy}^2- bu_{xy}^2 -bu_{yy}^2)\rho'\,dxdy -
\iint (u_{xy}^2 +u_{yy}^2)\rho'''\,dxdy \\ =
2\iint (uu_x)_y (u_{xxy}\rho +u_{xy} \rho' +u_{yyy}\rho)\,dxdy.
\end{multline}
We have:
\begin{multline*}
\iint uu_{xy}u_{xxy}\rho\,dxdy  \\ \leq \Bigl(\iint u^4\,dxdy\Bigr)^{1/4} \Bigl(\iint u_{xy}^4\rho^2\,dxdy\Bigr)^{1/4} \Bigl(\iint u^2_{xxy}\rho\,dxdy\Bigr)^{1/2} \\ \leq
\varepsilon \iint |D^2u_y|^2\rho\,dxdy + c(\varepsilon) \bigl(\|u\|^4_{H^1_+} +\|u\|_{H_+^1}\bigr)\iint
u_{xy}^2\rho\,dxdy,
\end{multline*}
\begin{multline*}
\iint u_xu_yu_{xxy}\rho\,dxdy \leq \sup\limits_{\Sigma_+}|u_y\rho^{1/2}|\Bigl(\iint u_x^2\,dxdy\Bigr)^{1/2} \Bigl(\iint u_{xxy}^2\rho\,dxdy\Bigr)^{1/2} \\ \leq
c\|u\|_{H_+^1}\iint |D^2u_y|^2\rho\,dxdy + c\|u\|^2_{H^1_+}\Bigl(\iint |D^2u_y|^2\rho\,dxdy\Bigr)^{1/2}.
\end{multline*}
Other terms in the right side of \eqref{5.12} can be handled in a similar way. Moreover,
$$
\int_0^L u^2_{xy}\big|_{x=0}\,dy \leq \varepsilon \iint u^2_{xxy}\rho\,dxdy + c(\varepsilon)\iint u^2_{xy}\rho\,dxdy.
$$
As a result, it follows from \eqref{5.8} and \eqref{5.12} that for $t\geq 0$ 
\begin{equation}\label{5.13}
e^{\alpha\beta t}\|u_y(t,\cdot,\cdot)\|^2_{H_+^{1,\exp(2\alpha x)}} + \int_0^t e^{\alpha\beta\tau} \iint |D^2u_y|^2\rho\,dxdy\,d\tau \leq c.
\end{equation}
Write down equality \eqref{1.1} in a form
\begin{equation}\label{5.14}
u_{xxx}= -u_t-bu_x-uu_x-u_{xyy} 
\end{equation}
Then inequalities \eqref{1.14} and \eqref{5.13} imply that
$$
e^{\alpha\beta t}\|u_{xx}(t,\cdot,\cdot)\|_{L_{2,+}^{\rho(x)}}^2 + \int_0^t e^{\alpha\beta\tau} \|u_{xxx}\|_{L_{2,+}^{\rho(x)}}^2\,d\tau \leq c.
$$
Combination of this inequality with \eqref{5.13} yields that
\begin{equation}\label{5.15}
e^{\alpha\beta t}\|u(t,\cdot,\cdot)\|^2_{H_+^{2,\exp(2\alpha x)}} + \int_0^t e^{\alpha\beta\tau} \iint |D^3u|^2\rho\,dxdy\,d\tau \leq c.
\end{equation}

Finally, write down equality \eqref{2.18}:
\begin{multline}\label{5.16}
\frac{d}{dt} \iint (u_{xyy}^2+u_{yyy}^2)\rho\,dxdy +
\int_0^L (\nu_2^2\rho +2\nu_2u_{xyy}\rho' -u_{xyy}^2\rho'' +bu_{xyy}^2\rho)\big|_{x=0}\,dy \\ + \iint (3u_{xxyy}^2 +4u_{xyyy}^2 +u_{yyyy}^2- bu_{xyy}^2 -bu_{yyy}^2)\rho'\,dxdy -
\iint (u_{xyy}^2 +u_{yyy}^2)\rho'''\,dxdy \\ =
2\iint (uu_x)_{yy} (u_{xxyy}\rho +u_{xyy} \rho' +u_{yyyy}\rho)\,dxdy.
\end{multline}
Here $(uu_x)_{yy} = uu_{xyy} +2u_yu_{xy} +u_xu_{yy}$ and
\begin{multline*}
 \iint uu_{xyy}u_{xxyy}\rho\,dxdy \leq \|u\|_{L_{\infty,+}} \|u_{xyy}\|_{L_{2,+}^{\rho(x)}} \|u_{xxyy}\|_{L_{2,+}^{\rho(x)}} \\ \leq \varepsilon \iint u_{xxyy}^2\rho\,dxdy +c(\varepsilon)\|u\|^2_{H_+^2} \|u\|^2_{H_+^{3,\rho(x)}},
\end{multline*}
\begin{multline*}
\iint u_yu_{xy}u_{xxyy}\rho\,dxdy \\ \leq \Bigl(\iint u_y^4\,dxdy\Bigr)^{1/4} \Bigl(\iint u_{xy}^4\rho^2\,dxdy\Bigr)^{1/4} \Bigl(\iint u^2_{xxyy}\rho\,dxdy\Bigr)^{1/2} \\ \leq \varepsilon \iint u_{xxyy}^2\rho\,dxdy +c(\varepsilon)\|u\|^2_{H_+^2} \|u\|^2_{H_+^{3,\rho(x)}}.
\end{multline*}
Moreover,
$$
\int_0^L u^2_{xyy}\big|_{x=0}\,dy \leq \varepsilon \iint u^2_{xxyy}\rho\,dxdy + c(\varepsilon)\iint u^2_{xyy}\rho\,dxdy.
$$
As a result, it follows from \eqref{5.15} and \eqref{5.16} that for $t\geq 0$ 
\begin{equation}\label{5.17}
e^{\alpha\beta t}\|u_{yy}(t,\cdot,\cdot)\|^2_{H_+^{1,\exp(2\alpha x)}} + \int_0^t e^{\alpha\beta\tau} \iint |D^2u_{yy}|^2\rho\,dxdy\,d\tau \leq c.
\end{equation}
Application of equality \eqref{5.14} yields that
$$
\|u_{xxx}(t,\cdot,\cdot)\|^2_{L_{2,+}^{\rho(x)}} \leq ce^{-\alpha\beta t}.
$$
Application of \eqref{1.15} finishes the proof.
\end{proof}

\end{document}